\documentclass{amsart}

\usepackage{amssymb}

\usepackage{comment}

\usepackage{stmaryrd}
\usepackage{multirow}
\usepackage[colorinlistoftodos]{todonotes}

\newtheorem{theorem}{Theorem}[section]
\newtheorem{lem}[theorem]{Lemma}
\newtheorem{prop}[theorem]{Proposition}

\newcounter{NoMain}

\newtheorem{mainthm}[NoMain]{Theorem}

\theoremstyle{definition}
\newtheorem{definition}[theorem]{Definition}

\theoremstyle{remark}
\newtheorem{remark}[theorem]{Remark}

\usepackage[bookmarks=true]{hyperref}

\numberwithin{equation}{section}

\newcommand{\R}{\mathbb{R}}
\newcommand{\N}{\mathbb{N}}

\newcommand{\C}{\mathbb{C}}

\newcommand{\fonction}[5]{#1: \begin{array}{ccc}
#2 & \rightarrow & #3 \\
 #4 & \longmapsto & #5 \end{array}}

\DeclareMathOperator{\Ima}{Im}
\DeclareMathOperator{\ad}{ad}
\DeclareMathOperator{\Ad}{Ad}
\DeclareMathOperator{\vspan}{span}
\DeclareMathOperator{\Id}{Id}
\DeclareMathOperator{\rank}{rank}
\DeclareMathOperator{\Aut}{Aut}
\DeclareMathOperator{\Stab}{Stab}

\DeclareMathOperator{\diam}{diam}

\begin{document}

\title{Cartan motion groups: regularity of $K$-finite matrix coefficients}

\author{Guillaume Dumas}
\address{Universite Claude Bernard Lyon 1, CNRS, Centrale Lyon, INSA Lyon, Université Jean Monnet, ICJ UMR5208, 69622
Villeurbanne, France.}

\date{\today}

\email{gdumas@math.univ-lyon1.fr}
\thanks{}

\begin{abstract}
If $G$ is a connected semisimple Lie group with finite center and $K$ is
a maximal compact subgroup of G, then the Lie algebra of $G$ admits a Cartan decomposition $\mathfrak{g}=\mathfrak{k}\oplus\mathfrak{p}$. This allows us to define the Cartan motion group $H=\mathfrak{p}\rtimes K$. In this paper, we study the regularity of $K$-finite matrix coefficients of unitary representations of $H$. We prove that the optimal exponent $\kappa(G)$ for which all such coefficients are $\kappa(G)$-Hölder continuous coincides with the optimal regularity of all $K$-finite coefficients of the group $G$ itself. Our approach relies on stationary phase techniques that were previously employed by the author to study regularity in the setting of $(G,K)$. Furthermore, we provide a general framework to reduce the question of regularity from $K$-finite coefficients to $K$-bi-invariant coefficients. 
\end{abstract}
\maketitle

\section{Introduction}
In his proof of his strengthening of property $(T)$ for $SL(3,\R)$, Lafforgue used the fact that $SO(2)$-invariant matrix coefficients of unitary representations of $SO(3)$ are $\frac{1}{2}$-Hölder continuous outside of some singular points (\cite{lafforgue}).

In \cite{dumas2023regularity, dumas2024regularity}, we obtained generalizations of this regularity result for many Lie groups. Given a Lie group $G$, the aim is to find some compact subgroup $K$ and the largest value of $\alpha>0$ such that all $K$-bi-invariant - and even $K$-finite - matrix coefficients of unitary representations of $G$ are of class $C^\alpha$.

As explained in \cite{dumas2023regularity}, a nice context to answer such questions is to assume that $(G,K)$ is a Gelfand pair. In that case, obtaining regularity results for $K$-bi-invariant coefficients reduces to the boundedness of a specific family in Hölder spaces, the family of spherical functions (see Section \ref{sec:gelfand}).

If $M$ is a Riemannian symmetric space, we can associate a natural Gelfand pair $(G,K)$ such that $M=G/K$, which we call a symmetric Gelfand pair (see Section \ref{sec:sym}). If $M$ is simply connected, then $M$ splits as a direct product $M=M_+\times M_-\times M_0$ where $M_+$ is of compact type (positive curvature), $M_-$ is of non-compact type (negative curvature) and $M_0$ is a Euclidean space (zero curvature).

Spaces of non-compact type correspond to quotient spaces $G/K$ where $G$ is a semisimple Lie group with finite center and $K$ a maximal compact subgroup. The regularity of their coefficients was studied in \cite{dumas2024regularity}. Such a space is always simply connected. If $U$ is a compact real form of $G$, then $U/K$ is a symmetric space of compact type, which we call the compact dual of $G/K$. The symmetric pairs $(U,K)$ with $U$ compact semisimple were studied in \cite{dumas2023regularity} and the results were improved in \cite{dumas2024regularity}.

The purpose of this note is to study the flat case. As pointed out in \cite[Lemma 2.9]{dumas2023regularity}, in compact and non-compact type there is essentially one connected symmetric Gelfand pair which represents $M$. This is far from being true in the Euclidean case. Indeed, consider $M=\R^n$. The group of displacement $G(M)$ is $\R^n$ itself, and the natural Gelfand pair is $(\R^n,\{0\})$. Then the spherical functions are just $x\mapsto e^{2i\pi \langle x,y\rangle}$ for $y\in \R^n$, which are not bounded in any Hölder space. On the other hand, the group of all automorphisms (of symmetric spaces) of $M$ is $G=\R^n\rtimes GL_n(\R)$, the group of affine map. For any compact subgroup $K$ of $GL_n(\R)$, we obtain a symmetric Gelfand pair $(\R^n\rtimes K,K)$. If for example $K=SO(n)$, the group $G$ is called the Euclidean motion group and the spherical functions of $G$ can be expressed in terms of Bessel functions (\cite[Chapter XI]{vilenkin1968special}). Thus, the behavior is very different in that case.

We will focus on pairs arising from Cartan motion groups. Let $G$ be a semisimple Lie group with finite center, then there is a Cartan decomposition $\mathfrak{g}=\mathfrak{k}\oplus \mathfrak{p}$ where $\mathfrak{k}$ is the Lie algebra of a maximal compact subgroup $K$. Then $K$ acts on $\mathfrak{p}$, seen as a vector space, by the adjoint representation. Let $H=\mathfrak{p}\rtimes K$. Then $(H,K)$ is a symmetric Gelfand pair and $H/K$ is a Euclidean symmetric space. Note that we can thus obtain three natural Gelfand pairs related to $G$: $(G,K)$ is a pair of non-compact type, $(U,K)$ is a pair of compact type for $U$ compact real form, and $(H,K)$ a flat pair.

Our goal in this paper is to find the optimal value $(r,\delta)\in \N\times [0,1]$ such that any $K$-finite matrix coefficient of $H$ belongs to the Hölder space $C^{(r,\delta)}(H_r)$ (see Section \ref{sec:holder} for the precise definition of these Hölder spaces). Here, $H_r$ is the dense open subset of regular points of $G$ (see Section \ref{sec:kak} for the definition). 

If $\mathfrak{a}$ is a maximal abelian subspace of $\mathfrak{p}$, we can consider the associated root system $\Sigma\subset \mathfrak{a}^*$ and a choice of positive roots $\Sigma^+$ (see Section \ref{sec:cartan} for more details, in particular for the definitions of those objects). 
For $\lambda\in \mathfrak{a}^*$, define \begin{equation*}n(\lambda)=\sum_{\begin{subarray}{c}
    \alpha\in \Sigma^+\\
    \langle \alpha,\lambda\rangle\neq 0
\end{subarray}} m(\alpha)\end{equation*}and set \begin{equation*}
\kappa(G)=\underset{\lambda \in  \mathfrak{a}^*\setminus \{0\}}{\inf} \frac{n(\lambda)}{2}.\end{equation*}

Our main result is the following.
\begin{mainthm}\label{mainthmA}
    Let $G$ be a connected semisimple Lie group with finite center and $H=\mathfrak{p}\rtimes K$ its Cartan motion group. Let $r=\lfloor \kappa(G)\rfloor$ and $\delta=\kappa(G)-r$. Then any $K$-finite matrix coefficient of a unitary representation of $H$ is in $C^{(r,\delta)}(H_r)$. Furthermore, this result is optimal in the sense that for any $\delta'>\delta$, there exists a $K$-finite (even $K$-bi-invariant) coefficient which is not in $C^{(r,\delta')}(H_r)$.
\end{mainthm}
Comparing with \cite[Thm. A and B]{dumas2024regularity}, the regularity obtained for $H$ on $H_r$ is exactly the regularity that was previously obtained for the semisimple group $G$, and the regularity that was obtained for its compact real form $U$ (only on a small open subset). Thus, all three Gelfand pairs arising from $G$ exhibit a similar behavior in terms of the regularity of their $K$-finite coefficients.\smallskip

We first explain how to reduce this question to understanding spherical functions of a Gelfand pair. In Section \ref{sec:gelfand} we show that spherical functions dictate the behavior of all $K$-bi-invariants. In Section \ref{sec:kfinite}, we prove that this is enough to understand all $K$-finite coefficients. Such a result was already proved in\cite[Section 5]{dumas2023regularity} and \cite[Section 3.2]{dumas2024regularity} in special cases (namely compact and non-compact type). Here, we give a general proof under some assumptions on the existence of a $KAK$ decomposition. Section \ref{sec:backgroundcartan} is devoted to preliminaries on Cartan motion groups and their harmonic analysis. We also prove some results on $KAK$ decomposition in this context (Section \ref{sec:kak}). Finally, we prove Theorem \ref{mainthmA} in Section \ref{sec:mainresult}. To do so, we proceed in a similar fashion to \cite{dumas2024regularity}. Indeed, spherical functions are expressed as an oscillatory integral, thus we may understand their asymptotics through the stationary phase approximation. The phase functions appearing in these integrals were studied in \cite{DKV}. 

\subsection*{Acknowledgments} I am grateful to Alexandre Afgoustidis and Jean-Louis Clerc who pointed out the question of Cartan motion groups to me. I would like to thank my Ph.D. advisor Mikael de la Salle for his involvement.

\section{Reduction to harmonic analysis on a Gelfand pair}
\subsection{Hölder spaces}\label{sec:holder}
\begin{definition}\label{def:holdermultivar}Let $(X,d)$ be a metric space and $U$ open subset of $X$, $(E,\Vert.\Vert)$ a normed vector space, $\alpha\in ]0,1]$. A function $f:U\to E$ is $\alpha$-Hölder if for any compact subset $K$ of $U$, there is $C_K>0$ such that $\forall x,y\in K$, $\Vert f(x)-f(y)\Vert \leq C_Kd(x,y)^\alpha$.

If $X$ is also a normed vector space and $r\in \N$, we say that the map $f$ belongs to $C^{(r,\alpha)}(U,E)$ if $f\in C^r(U,E)$ and the $r$-th differential $D^rf$ is $\alpha$-Hölder as a map from $U$ to the vector space of multilinear $r$-forms. We extend to $\alpha=0$ by $C^{(r,0)}(U,E)=C^r(U,E)$.

For $K$ a compact subset of $U$ and $f\in C^{(r,\alpha)}(U,E)$, define $$\Vert f\Vert_{C^{(r,\alpha)}(K,E)}=\max \left\{\underset{k\leq r}{\max}\, \underset{x\in K}{\sup}\Vert D^kf(x)\Vert,\underset{x,y\in K,x\neq y}{\sup}\frac{\Vert D^rf(x)-D^rf(y)\Vert}{d(x,y)^\alpha}\right\}.$$
The family of semi-norms $\Vert.\Vert_{C^{(r,\alpha)}(K,E)}$ for $K$ a compact subset of $U$ makes the space $C^{(r,\alpha)}(U,E)$ into a Fréchet space.

Finally if $(X,d)$ is a Riemannian manifold, we say that $f\in C^{(r,\alpha)}(U,E)$ if for any chart $(\varphi,V)$ of $U$, $f\circ \varphi^{-1}\in C^{(r,\alpha)}(\varphi(V),E)$.
\end{definition}
\begin{remark}\label{rem:holder}
    If $U$ is locally compact, a function $f:U\to E$ is $\alpha$-Hölder if and only if for any $x\in U$, there exists a neighborhood $U_x$ of $x$ and a constant $C_x>0$ such that for any $y,z\in U_x$, $\Vert f(y)-f(z)\Vert\leq C_xd(y,z)^\alpha$.
\end{remark}
We will denote $C^{(r,\alpha)}(U,\C)$ by $C^{(r,\alpha)}(U)$.

The following lemma will be useful throughout the article and can be found in \cite[Lemma 2.1]{dumas2023regularity}
\begin{lem}
\label{lem:precomposition}Let $(X,d)$ and $(Y,d')$ be two Riemannian manifolds and $U,V$ open subsets of $X,Y$ respectively. Let $\alpha>0$ and $r\in \N$. Let $\varphi:U\to V$ be a function of class $C^{\infty}$. Then $\varphi_*:f\mapsto f\circ \varphi$ maps $C^{(r,\alpha)}(V)$ to $C^{(r,\alpha)}(U)$ and is continuous.
\end{lem}

\subsection{Gelfand pairs}\label{sec:gelfand}
\begin{definition}Let $G$ be a locally compact topological group with a left Haar measure $dg$ and $K$ a compact subgroup with normalized Haar measure $dk$. The pair $(G,K)$ is a Gelfand pair if the algebra of continuous $K$-bi-invariant functions on $G$ with compact support is commutative for the convolution.

A spherical function of $(G,K)$ is a continuous $K$-bi-invariant non-zero function on $G$ such that for all $x,y\in G$, $$\int_K \varphi(xky)\,dk=\varphi(x)\varphi(y).$$
\end{definition}

A standard result (see \cite[Coro. 6.3.3]{Dijk+2009}) gives a link between spherical functions of $(G,K)$ and unitary representations of $G$.
\begin{prop}
If $(G,K)$ is a Gelfand pair, then for any irreducible unitary representation $\pi$ of $G$ on a Hilbert space $\mathcal{H}$, the subspace $\mathcal{H}^K$ of $K$-invariant vectors is of dimension at most $1$.

The positive-definite spherical functions of $G$ are exactly the matrix coefficients $g\mapsto \langle\pi(g)v,v\rangle$ with $\pi$ an irreducible unitary representation of $G$ and $v$ a $K$-invariant unit vector.

If $G$ is compact, any spherical function is positive-definite.
\end{prop}

More details on Gelfand pairs can be found in \cite[Ch. 5,6,7]{Dijk+2009}.\smallskip

Given a Gelfand pair $(G,K)$, it is natural to study spherical functions in order to get results on $K$-bi-invariant matrix coefficients of unitary representations. Indeed, any matrix coefficient of a unitary representation decomposes into an integral of spherical functions - an infinite sum if $G$ is compact. Then studying boundedness of positive-definite spherical functions in some Hölder spaces is enough to obtain regularity for all $K$-bi-invariant matrix coefficients of unitary representations. More precisely, the optimal regularity of such coefficients is exactly the optimal uniform regularity of spherical functions. The proof of the following two lemmas can be found in \cite[Section 2.2]{dumas2023regularity}.
\begin{lem}\label{lem:decinteg}
Let $(G,K)$ be a Gelfand pair with $G$ second countable. Let $\varphi$ be a $K$-bi-invariant matrix coefficient of a unitary representation $\pi$ on an Hilbert space $\mathcal{H}$. Then, there exists a standard Borel space $X$ and a $\sigma$-finite measure $\mu$ on $X$ such that $$\varphi=\int_X c_x\varphi_x d\mu(x)$$where $\varphi_x$ is a positive-definite spherical function of $(G,K)$ for any $x\in X$ and $c\in L^1(X,\mu)$.
\end{lem}

\begin{lem}\label{lem:lienspheriquekbiinv}Let $(G,K)$ be a Gelfand pair with $G$ a Lie group endowed with a Riemannian metric $d$ and $U$ any open subset of $G$. Let $(\varphi_\lambda)_{\lambda\in \Lambda}$ be the the family of positive-definite spherical functions of $(G,K)$. Then $(\varphi_\lambda)_{\lambda\in \Lambda}$ is bounded in $C^{(r,\delta)}(U)$ if and only if any $K$-bi-invariant matrix coefficient of a unitary representation of $G$ is in $C^{(r,\delta)}(U)$.
\end{lem}

\subsection{From \texorpdfstring{$K$}{K}-bi-invariant to \texorpdfstring{$K$}{K}-finite coefficients}\label{sec:kfinite} Let $G$ be a Lie group and $K$ a compact subgroup of $G$. 

\begin{definition}Let $\pi$ be a unitary representation of $G$ on $\mathcal{H}$ and $(\rho,V)$ a representation of $K$. We say that $\xi\in\mathcal{H}$ is \begin{itemize}
    \item $K$-finite if $\vspan(\pi(K)\xi)$ is finite dimensional,
    \item of $K$-type $V$ if $\vspan(\pi(K)\xi)\simeq V$ as a representation of $K$.
\end{itemize}
\end{definition}
Note that this definition of $K$-type $V$ is not standard.\medskip

The aim of this subsection is to prove that a regularity result on the space of all \textit{$K$-bi-invariant} coefficients of $G$ implies that the same regularity holds for all \textit{$K$-finite} coefficients of $G$. This was done in \cite[Section 5]{dumas2023regularity} for compact symmetric pairs, and then extended in \cite[Section 3.2]{dumas2024regularity} in the non-compact case. We here give a unified proof which holds under some mild assumptions on the pair $(G,K)$.

The first step is a lemma which first appeared in \cite[Lemma 2.2]{dLMdlS} in the case of compact groups, which will allow to smoothly transform a $K$-finite coefficient into invariant coefficients. The original proof relied on the fact that compact group have a rich finite dimensional representation theory, and that these representations are smooth. In general, we cannot use finite dimensional unitary representations, however it remains true that $G$ has a rich set of smooth coefficients.

If $\pi$ is a unitary representation of $G$ on an Hilbert space $H$, we say that $\xi\in H$ is a smooth vector if $g\mapsto \pi(g)\xi$ is a smooth map from $G$ to $H$. Clearly, matrix coefficients associated to smooth vector are smooth. Consider the Gårding subspace of $\pi$,
denoted $H_\infty$, which is the subspace of $H$ spanned by all vectors of the form $$\pi(f)\xi=\int_G f(g)\pi(g)\xi dg$$for $\xi\in H,f\in C^\infty_c(G)$, where $dg$ is a left-invariant Haar measure. Then it is well-known that $H_\infty$ is dense in $H$ and every vector of $H_\infty$ is smooth (\cite[Prop. 3.14, Thm. 3.15]{knapp2001representation}). In fact, it is a deep result of Dixmier and Malliavin (\cite{dixmier1978factorisations}) that the Gårding subspace coincides with the space of smooth vectors, but we do not need this.

Let $U=K\times K$. Let $(\rho,E)$ be a finite-dimensional unitary representation of $U$. For $g\in G$, let $U_g$ be the stabilizer of $g$ under the action of $U$ by left-right multiplication, i.e. $$U_g=\{(k_1,k_2)\in U\vert\ k_1gk_2^{-1}=g\}.$$Let $E_g=E^{U_g}$ the subspace of vectors fixed by $U_g$, and $P_g:E\to E_g$ the orthogonal projection. 

\begin{lem}\label{lem:smoothkfintokinv}
For any $g_0\in G$, there exists a smooth function $\psi:G\to B(E)$
 such that \begin{enumerate}
     \item $\forall u\in U,g\in G$, $\psi(u.g)=\psi(g)\circ \rho(u)^{-1}$,
     \item $\forall v_1,v_2\in E$, $g\mapsto \langle \psi(g)v_1,v_2\rangle$ is a matrix coefficient of a unitary representation of $G$,
     \item $\psi(g_0)=P_{g_0}$.
 \end{enumerate}
\end{lem}

\begin{proof}
    Let $F$ be the set of functions $\phi:G\to B(V)$ such that for any $v_1,v_2\in E$, the map $g\mapsto \langle \psi(g)v_1,v_2\rangle$ is a matrix coefficients of a unitary representation of $G$ with vectors in the Gårding subspace of that representation. Then such a map $F$ is smooth and verifies $(2)$.

    If $\phi\in F$, define $\psi(g)=\int_U \phi(u.g)\rho(u)du$ where $du$ is the Haar probability measure of $U$. Since $U$ is compact and $\phi$ smooth, $\phi$ is bounded on the compact orbit $U.g$ so the integral makes sense. Then a simple change of variables show that $\psi$ verifies $(1)$. Let $e_1,\dots,e_d$ be an orthonormal basis of $E$. Let $\pi_{ij}$ be unitary representations of $G$ and $\xi_{ij},\eta_{ij}$ vectors in the Gårding subspace of $\pi_{ij}$ such that for all $1\leq i,j\leq d$ and $g\in G$,$$\langle \phi(g)e_i,e_j\rangle=\langle \pi_{ij}(g)\xi_{ij},\eta_{ij}\rangle.$$
    Let $u=(k_1,k_2)\in U$, then $\rho(u)=\rho(k_1,1_K)\rho(1_K,k_2)$. We define functions $K\to \mathbb{C}$ such that $$\rho(1_K,k)e_i=\sum_{j=1}^d \lambda_{ij}(k)e_j$$and$$\rho(k,1_K)e_i=\sum_{j=1}^d \mu_{ij}(k)e_j.$$Then \begin{multline*}
   \langle \psi(g)e_i,e_j\rangle   \\ \begin{aligned}&= \int_U \langle \phi(u.g)\rho(u)e_i,e_j\rangle du \\
     & =  \int_{K\times K}\sum_{p,q=1}^d \mu_{pq}(k_1)\lambda_{ip}(k_2) \langle \phi(k_1gk_2^{-1})e_q,e_j\rangle dk_1dk_2\\
     & =  \int_{K\times K}\sum_{p,q=1}^d \mu_{pq}(k_1)\lambda_{ip}(k_2) \langle \pi_{qj}(k_1gk_2^{-1})\xi_{qj},\eta_{qj}\rangle dk_1dk_2\\
     & = \sum_{p,q=1}^d\int_{K\times K} \langle \pi_{qj}(g)\left(\lambda_{ip}(k_2)\pi_{qj}(k_2^{-1}) \xi_{qj}\right),\overline{\mu_{pq}(k_1)}\pi_{qj}(k_1^{-1})\eta_{qj}\rangle dk_1dk_2\\
     & =  \sum_{p,q=1}^d \left\langle \pi_{qj}(g)\left(\int_K \lambda_{ip}(k_2)\pi_{qj}(k_2^{-1})\xi_{qj}dk_2\right),\left(\int_K \overline{\mu_{pq}(k_1)}\pi_{qj}(k_1^{-1})\eta_{qj} dk_1\right) \right\rangle.
     \end{aligned}
\end{multline*}
Since $\xi_{qj}$ is in the Gårding subspace of $\pi_{qj}$, it is a finite linear combination of vectors of the form $$\pi_{qj}(f)\xi=\int_G f(g)\pi_{qj}(g)\xi \ dg.$$Then by left invariance of $dg$, \begin{align*}
    \int_K \lambda_{ip}(k)\pi_{qj}(k^{-1})\pi_{qj}(f)\xi\ dk & = \int_K\int_G \lambda_{ip}(k)f(g)\pi_{qj}(k^{-1}g)\xi \ dg dk \\
    & = \int_K\int_G  \lambda_{ip}(k) f(kg) \pi_{qj}(g)\xi \ dgdk\\
    & = \int_G \left(\int_K \lambda_{ip}(k) f(kg) dk \right) \pi_{qj}(g)\xi \ dg
\end{align*}so $ \int_K \lambda_{ip}(k)\pi_{qj}(k^{-1})\pi_{qj}(f)\xi\ dk$ is in the Gårding subspace and by linearity, so is $\int_K \lambda_{ip}(k_2)\pi_{qj}(k_2^{-1})\xi_{qj}dk_2$. Similarly, $\int_K \overline{\mu_{pq}(k_1)}\pi_{qj}(k_1^{-1})\eta_{qj} dk_1$ is in the Gårding subspace of $\pi_{qj}$. Thus, $\langle \psi(g)e_i,e_j\rangle$ is a matrix coefficient of $\bigoplus_{q=1}^d \pi_{qj}^{\oplus d}$ whose vectors are in the Gårding subspace. By linearity, this is true for $v_1,v_2\in E$. Thus, we showed that if $\phi\in F$, then $\psi\in F$, hence $\psi$ is smooth and verifies $(2)$.

It remains to show that there exists $\phi\in F$ such that $\psi(g_0)=P_{g_0}$. Notice that if $u\in U_{g_0}$, then $$\psi(g_0)=\psi(u.g_0)=\psi(g_0)\circ \rho(u)^{-1}.$$
Thus $E_{g_0}^\bot=\sum_{u\in U_{g_0}}\Ima(\rho(u)-\Id)\subset \ker \psi(g_0)$. Thus, condition $(3)$ is a "maximal rank" condition.

First, let us find $\phi\in F$ such that $\rank\psi(g_0)=\dim E_{g_0}$. Consider $O\simeq U/U_{g_0}$ the $U$-orbit of $g_0$ in $G$. Let $s$ be a measurable section, that is to say $s:O\mapsto U$ such that $s(u.g_0)\in uU_{g_0}$. Let $\phi:O\mapsto B(E)$ be the map $x\mapsto \rho(s(x))^{-1}$. Then $\psi:x\mapsto \int_U \rho(s(u.x)^{-1}u)du$ is such that $\psi(g_0)$ is the identity on $E_{g_0}$, and by the above discussion $0$ on $E_{g_0}^\bot$. Thus $\psi(g_0)=P_{g_0}$. Let $\mu$ be the image of the Haar measure on $O$ by the map $p:u\mapsto u.g_0$. Then $\phi\in L^1(O;B(E),\mu)$. By density of continuous function, there are continuous maps $f:O\to B(E)$ arbitrarily close to $\phi$ in $\Vert.\Vert_1$. But then, \begin{align*}
    \left\Vert \int_U f(u.g_0)\rho(u)du-\int_U \phi(u.g_0)\rho(u)du \right\Vert & \leq  \int_U \Vert f(u.g_0)-\phi(u.g_0)\Vert du  \\
     & \leq  \int_U \Vert (f-\phi)\circ p\Vert du\\
     & \leq  \int_O \Vert f-\phi\Vert d\mu\\
     & \leq  \Vert f-\phi\Vert_1.
\end{align*}

So we can take $f$ close enough so that $\rank \int_U f(u.g_0)\rho(u)du=\rank P_{g_0}$. Then since $O$ is closed in $G$ normal, by Tietze extension theorem, we can extend $f$ to a continuous map $\phi:G\mapsto B(E)$.

Let $L$ be a compact subset of $G$ containing $g_0$ and $\varepsilon>0$. Let $\phi_{ij}:g\mapsto \langle \phi(g)e_i,e_j\rangle$. By the Gelfand-Raikov's theorem, there exists $\pi_{ij}$ a unitary representation of $G$ and a matrix coefficient $\varphi_{ij}$ of $\pi_{ij}$ such that $$\underset{g\in L}{\sup} \vert \phi_{ij}(g)-\varphi_{ij}(g)\vert \leq \varepsilon.$$Then by density of the Gårding subspace, we may find a coefficient $\Tilde{\varphi}_{ij}$ of $\pi_{ij}$ whose vectors are in the Gårding subspace of $\pi_{ij}$ and such that $$\underset{g\in L}{\sup} \vert \varphi_{ij}(g)-\Tilde{\varphi}_{ij}(g)\vert \leq \varepsilon.$$
Then, define $\varphi:G\to B(E)$ by $\varphi(g)e_i=\sum_{j=1}^d \Tilde{\varphi}_{ij}(g)e_j$, by construction $\varphi\in F$. So we can find find $\varphi$ in $F$ arbitrarily close to $\phi$ on any compact subset containing $g_0$, in particular on the orbit $O$. Thus, for $\varepsilon$ small enough, $\int_{U} \varphi(u.g_0)\rho(u)du$ is of rank $\dim E_{g_0}$.

Finally, we get $\phi\in F$ such that $\psi(g_0)$ has rank $\dim E_{g_0}$ and is zero on $E_{g_0}^\bot$. Thus there is $A\in B(E)$ such that $A\psi(g_0)=P_{g_0}$. Replace $\phi$ by $A\phi$ and we get the result.
\end{proof}

With this lemma in hand, we can now prove regularity for $K$-invariant coefficients. However, although Lemma \ref{lem:smoothkfintokinv} holds without assumption, we need to assume that the pair $(G,K)$ has some sort of a $KAK$ decomposition to carry on - in the spirit of the well-known $KAK$ decomposition of semisimple Lie groups.

\begin{definition}\label{def:smoothkak} Let $G$ be a Lie group and $K$ a compact subgroup of $G$. Let $G_r$ be an open subset of $G$ which is invariant by multiplication on the left and right by elements of $K$. We say that $(G_r,K)$ has a well-behaved $KAK$ decomposition if there exists a submanifold $A_r$ of $G$ and a subgroup $M$ of $K$ such that:
\begin{itemize}
    \item for any $g\in G_r$, there exists a unique $a(g)\in A_r$ such that $g\in Ka(g)K$;
    \item the map $a:G_r\to A_r$ is smooth;
    \item the stabilizer of $a\in A_r$ under the left-right multiplication action of $K$ is $\Delta(M)=\{(k,k) \vert \ k\in M \}$ and thus independent of $a$;
    \item for any $x\in G_r$, there exists a neighborhood $O_x$ and a choice of $g\mapsto k_i(g)$ defined on $O_x$, $i=1,2$ such that $k_i$ is smooth and for any $g\in O_x$, $g=k_1(g)a(g)k_2(g)^{-1}$ 
\end{itemize}
\end{definition}
The existence of a $KAK$ decomposition of a semisimple Lie group is well-known, and it is proven in \cite[Prop. 3.3]{dumas2024regularity} that it is well-behaved on a specific dense open subset. A similar result holds for compact semisimple Lie groups (\cite[Prop. 5.8]{dumas2023regularity}) and we will show this for Cartan motion groups in Section \ref{sec:kak}.\medskip

Let $\pi$ be a unitary representation of $G$ on $\mathcal{H}$ and $\xi,\eta\in\mathcal{H}$ of $K$-type $V,W$ respectively, for $V,W$ irreducible representations of $K$. Denote $V_\xi=\vspan(\pi(K)\xi)$. Then there is an isomorphism $i_\xi:V\to V_\xi\subset \mathcal{H}$, denote $\xi_0=i_\xi^{-1}(\xi)$. Similarly, define $V_\eta$, $i_\eta$ and $\eta_0$. Then the map \begin{equation}\label{eq:defi_f}\fonction{f}{B(\mathcal{H})}{L(V,W^*)\simeq V^*\otimes W^*}{A}{i_\eta^*Ai_\xi}\end{equation}is $K\times K$ equivariant.

For the associated matrix coefficient, we have $\varphi(g)=\langle\pi(g)\xi,\eta\rangle=\langle f(\pi(g))\xi_0,\eta_0\rangle$.

Now denote $(\rho,E)$ the irreducible representation of $U=K\times K$ on $V^*\otimes W^*$. The $U$-equivariance of $f$ means that for any $(k,k')\in U$ and $A\in B(\mathcal{H})$, we have \begin{equation}\label{eq:equiv_f}f(\pi(k)A\pi(k')^{-1})=\rho(k,k')(f(A)).\end{equation}Furthermore, there are $v_1,\cdots,v_n\in E$ and $\xi_1,\cdots,\xi_n,\eta_1,\cdots,\eta_n\in\mathcal{H}$ such that\begin{equation}\label{eq:dec_f}f(A)=\sum_{i=1}^n \langle A\xi_i,\eta_i\rangle v_i.\end{equation}

\begin{prop}\label{prop:regktypeV}Let $G_r$ be an open subset of $G$ such that $(G_r,K)$ has a well-behaved $KAK$-decomposition. Assume that any $K$-bi-invariant matrix coefficient $\varphi$ of a unitary representation of $G$ is in $C^{(r,\delta)}(G_r)$, then the map $f\circ \pi$ is in $C^{(r,\delta)}(G_r,E)$.
\end{prop}
\begin{proof}
Let $g_0\in G_r$ and $\psi$ given by Lemma \ref{lem:smoothkfintokinv} for the representation $(\rho,E)$. Let $\Tilde{f}:g\mapsto \psi(g)(f(\pi(g)))$. By \eqref{eq:equiv_f} and $(1)$ of Lemma \ref{lem:smoothkfintokinv}, we have \begin{equation}\label{eq:kbiinv_tildef}\Tilde{f}(u.g)=\psi(u.g)(f(\pi(u.g)))=\psi(g)\rho(u)^{-1}\rho(u)(f(\pi(g))=\Tilde{f}(g)\end{equation}
so $\Tilde{f}$ is a $K$-bi-invariant map.

Let $(e_1,\cdots,e_d)$ be an orthonormal basis of $V_\rho$, by $(2)$ of Lemma \ref{lem:smoothkfintokinv} there are $(\pi_{ij},\mathcal{H}_{ij})$ unitary representations of $G$ and $a_{ij},b_{ij}\in \mathcal{H}_{ij}$ such that $$\langle\psi(g)v_i,e_j\rangle=\langle\pi_{ij}(g)a_{ij},b_{ij}\rangle$$so $\psi(g)v_i=\sum_{j=1}^d \langle\pi_{ij}(g)a_{ij},b_{ij}\rangle e_j$ and finally with \eqref{eq:dec_f}, \begin{equation}\label{eq:tildef}\Tilde{f}(g)=\sum_{i=1}^n\sum_{j=1}^d \langle(\pi_{ij}\otimes \pi)(g)(a_{ij}\otimes\xi_i),b_{ij}\otimes \eta_i\rangle e_j.\end{equation}

Hence, $\Tilde{f}$ is a sum of $K$-bi-invariant matrix coefficients of unitary representations of $G$, so by the hypothesis, $\Tilde{f}\in C^{(r,\delta)}(G_r,E)$.\smallskip

Since $(G_r,K)$ has a well-behaved $KAK$ decomposition, consider $A_r,M$ from Definition \ref{def:smoothkak}. Then for any $a\in A_r$, $U_a=\Delta(M)$. Thus, $E_a=E_0$ is independent of $a\in A_r$. If $g=(k_1,k_2).a=k_1ak_2^{-1}$, we have $(k,k')\in U_g$ if an only if $(k_1^{-1}kk_1,k_2^{-1}k'k_2)\in \Delta(M)$ and so $E_g=\rho(k_1,k_2)E_0$.

Let $g_0=k_0a_0k_0^{'-1}$ and $E_1=E_{g_0}$. Since $\psi(g_0)=P_{g_0}$, there is an orthonormal basis adapted to $E_1$ such that $$\psi(g_0)=\begin{pmatrix}\Id&0\\0&0\end{pmatrix}.$$ Furthermore, since $\psi$ is smooth, there is $A_{g_0}$ neighborhood of $g_0$ such that $$\psi(g)=\begin{pmatrix}A(g)&*\\ *&*\end{pmatrix}$$
with $g\mapsto A(g)$ smooth, $A(g)$ invertible for any $g\in A_{g_0}$. Up to restricting $A_{g_0}$, by assumption on $(G_r,K)$, we have $g=k_1(g)a(g)k_2(g)^{-1}$ with $k_1,k_2$ smooth on $A_{g_0}$.

By \eqref{eq:kbiinv_tildef}, for any $g\in A_{g_0}$, we have $$\Tilde{f}(g)=\Tilde{f}(a(g))=\Tilde{f}(k_0a(g)k_0^{'-1}).$$
But then $f(\pi(k_0a(g)k_0^{'-1}))\in E_{k_0a(g)k_0^{'-1}}=\rho(k_0,k_0')E_0=E_1$. Set $$\Phi(g)=\rho(k_1(g)k_0^{-1},k_2(g)k_0^{'-1})\begin{pmatrix}
    A(k_0a(g)k_0^{'-1})^{-1} & 0\\0&0
\end{pmatrix},$$ it is a smooth map on $A_{g_0}$ because $A$ is smooth invertible, $k_1,k_2$ are smooth and $\rho$ is a finite dimensional representation of $U$ thus smooth. Since $f(\pi(k_0a(g)k_0^{'-1}))\in E_1$, we have \begin{align*}
    \Phi(g)(\Tilde{f}(g)) & =  \Phi(g)(\Tilde{f}(\pi(k_0a(g)k_0^{'-1})))  \\
     & =  \Phi(g)\psi(g)(f(\pi(k_0a(g)k_0^{'-1})))\\
     & =  \rho(k_1(g)k_0^{-1},k_2(g)k_0^{'-1})(f(\pi(k_0a(g)k_0^{'-1})))\\
     & =  f(\pi(k_1(g)a(g)k_2(g)^{-1}))\\
     & =  f(\pi(g))
\end{align*}

Now let $B:B(V)\times V\to V$ be the bilinear map sending $(u,v)$ to $u(v)$. We showed that on $A_{g_0}$, $f\circ \pi=B\circ (\Phi,\Tilde{f})$. Since $\Phi$ is smooth on $A_{g_0}$ and $\Tilde{f}\in C^{(r,\delta)}(G_r,E)$, we get by Leibniz formula that $f\circ g\in C^{(r,\delta)}(A_{g_0},E)$.

So for any $g_0\in G_r$, there exists a neighborhood $A_{g_0}$ such that $f\circ \pi \in C^{(r,\delta)}(A_{g_0},E)$. Thus, $f\circ \pi \in C^{(r,\delta)}(G_r,E)$.
\end{proof}

\begin{theorem}\label{thm:kfinite}Let $G$ be a Lie group, $G_r$ be an open subset of $G$ and $K$ a compact subgroup of $G$ such that $(G_r,K)$ has a well-behaved $KAK$-decomposition. The optimal regularity of $K$-bi-invariant matrix coefficient of unitary representations of $G$ on $G_r$ is equal to the optimal regularity of $K$-finite matrix coefficients of unitary representations of $G$ on $G_r$.
\end{theorem}
\begin{proof}One inequality is trivial since $K$-bi-invariant coefficients are $K$-finite.

For the other inequality, let $(r,\delta)$ such that any $K$-bi-invariant matrix coefficient of unitary representations $G$ is in $C^{(r,\delta)}(G_r)$. Let $\varphi:g\mapsto \langle \pi(g)\xi,\eta \rangle$ be a $K$-finite matrix coefficient of a unitary representation.

If $\xi,\eta$ are of $K$-type $V,W$ respectively, with $V,W$ irreducible representations of $K$, we showed that $\varphi(g)=\langle\pi(g)\xi,\eta\rangle=\langle f(\pi(g))\xi_0,\eta_0\rangle$ and in Proposition \ref{prop:regktypeV} that $f\circ \pi\in C^{(r,\delta)}(G_r)$, thus $\varphi\in C^{(r,\delta)}(G_r)$.

For the general case, if $\xi,\eta$ are $K$-finite, $V_\xi,V_\eta$ are finite dimensional representations of $K$, so they decompose into a finite number of irreducible representations. Thus, $\varphi$ is a finite sum of matrix coefficient of the previous case, so $\varphi\in C^{(r,\delta)}(G_r)$.
\end{proof}

\section{Cartan motion groups}\label{sec:backgroundcartan}
\subsection{Flat symmetric spaces and Gelfand pairs}\label{sec:sym}
\begin{definition}\label{def:symspace}
    A symmetric space is a smooth manifold $M$ together with a smooth map $\mu:M\times M\to M$, denoted $\mu(x,y)=x\cdot y$, which verifies the following properties:\begin{enumerate}
    \item $\forall x\in M$, $x\cdot x=x$;
    \item $\forall x,y\in M$, $x\cdot (x\cdot y)= y$;
    \item $\forall x,y,z\in M$, $x\cdot (y\cdot z)=(x\cdot y)\cdot (x\cdot z)$;
    \item $\forall x\in M$, there is a neighborhood $U\subset M$ of $x$ such that if $y\in U$, $x\cdot y=y \Rightarrow x=y$.\end{enumerate}
\end{definition}
Let $G$ be a connected Lie group, $\sigma$ an involutive automorphism of $G$ and $G^\sigma$ the subgroup of fixed points of $\sigma$. If $K$ is a subgroup of $G$ such that $(G^\sigma)_0\subset K\subset G^\sigma$, then the quotient space $M=G/K$ carries a natural structure of symmetric space with $xK\cdot yK=x\sigma(x)^{-1}\sigma(y)K$ (\cite{loos1969symmetric}). When $K$ is compact, $(G,K)$ is a Gelfand pair, which we call a symmetric Gelfand pair, and $M$ is Riemannian.

Conversely given a Riemannian symmetric space $M$ and $o\in M$, there exists a canonical connected Lie group $G(M)$ called the group of displacements of $M$, which is a subgroup of the group $\mathrm{Aut}(M)$ of all automorphisms of symmetric space of $M$. Let $K(M)$ be the isotropy group of $o$ which is compact, then $(G(M),K(M))$ is a symmetric pair and $M\simeq G(M)/K(M)$ as a symmetric space.

If $M$ is a simply connected symmetric space, there are $M_0,M_c,M_{nc}$ respectively Euclidean, of compact type and of non-compact type such that $M=M_0\times M_c\times M_{nc}$. We say that $M$ is semisimple if $G(M)$ is a semisimple Lie group.

The following lemma can be found in \cite[Lemma 2.9]{dumas2023regularity}
\begin{lem}
\label{lem:indepgelfandpair}Let $(G,K)$ be a symmetric pair and $M=G/K$ the associated symmetric space. If $M$ is semisimple, then there is a bijection between spherical functions of $(G,K)$ and spherical functions of $(G(M),K(M))$, such that the image of $\varphi$ induces the same function as $\varphi$ on $M$.
\end{lem}
This lemma essentially means that when $M$ is semisimple, the behavior of spherical functions (and thus, the regularity of $K$-bi-invariant matrix coefficients of $G$) does not depend on the choice of a symmetric pair representing $M$.

When $M$ is of compact or non-compact type, then $M$ is semisimple. The associated pairs were studied in \cite{dumas2023regularity,dumas2024regularity}. On the other hand, for a Euclidean symmetric space, there is a variety of pairs representing the space which exhibit different behaviors.

Let $n\in \N$, and consider $M=\R^n$. Then $G(M)=\R^n$ and $\Aut(M)=\R^n \rtimes GL_n(\R)$. Let $K$ be a compact subgroup of $GL_n(\R)$ and $H=\R^n \rtimes K$, then $(H,K)$ is a symmetric Gelfand pair (with the involution $\sigma(x,g)=(-x,g)$ for any $(x,g)\in H$). Clearly, when $K$ is the trivial group, there exist coefficients of unitary representations of $\R^n$ that are continuous, but not Hölder continuous (for any exponent). However, it is interesting to ask what happens when $K$ is non-trivial. $\lambda$

\subsection{Cartan motion groups}\label{sec:cartan}
Let $G$ be a connected real semisimple Lie group with finite center and $\mathfrak{g}$ its Lie algebra. Let $\theta$ be a Cartan involution of $\mathfrak{g}$ and $\mathfrak{g}=\mathfrak{k}\oplus \mathfrak{p}$ be the decomposition of $\mathfrak{g}$ in $\pm 1$-eigenspaces of $\theta$. Then $K=\exp \mathfrak{k}$ is a maximal compact subgroup of $G$. The subspace $\mathfrak{p}$ is stable by $\Ad(k)$ for each $k\in K$, thus we can define the semi-direct product $H=\mathfrak{p}\rtimes K$, called the Cartan motion group associated to $G$. As explained in Section \ref{sec:sym}, $(H,K)$ is a symmetric Gelfand pair and $H/K$ is isomorphic to the Euclidean symmetric space $\mathfrak{p}$.  

Consider $\mathfrak{a}$ a maximal abelian subspace of $\mathfrak{p}$. Let $\ell=\rank G=\dim \mathfrak{a}$. For $\alpha\in \mathfrak{a}^*$, define $\mathfrak{g}_\alpha=\{X\in \mathfrak{g}\vert \forall H\in \mathfrak{a}, [H,X]=\alpha(H)X\}$ the root space associated to $\alpha$. Let $m(\alpha)=\dim(\mathfrak{g}_\alpha)$ and $\Sigma=\{\alpha\neq 0 \vert m(\alpha)\geq 1\}$ be the set of roots. We say that $\Sigma$ is the restricted root system of $G$. Let $\mathfrak{m}=\mathfrak{g}_0\cap \mathfrak{k}$. Then the Lie algebra of $G$ decomposes as $$\mathfrak{g}=\mathfrak{m}\oplus \mathfrak{a}\oplus \bigoplus_{\alpha\in \Sigma} \mathfrak{g}_\alpha.$$

The Killing form of $\mathfrak{g}$ induces an inner product on $\mathfrak{a}$, denoted $\langle \cdot,\cdot\rangle$. Then for $\lambda\in \mathfrak{a}^*$, there is a unique $H_\lambda\in \mathfrak{a}$ such that for any $H\in \mathfrak{a}$, $\lambda(H)=\langle H_\lambda,H\rangle$. We use the isomorphism $\lambda \mapsto H_\lambda$ to define an inner product on $\mathfrak{a}^*$ by $$\langle \lambda,\mu\rangle=\langle H_\lambda,H_\mu\rangle.$$

Let $W$ be the Weyl group of the root system $\Sigma$, which is the subgroup of $O(\mathfrak{a}^*)$ generated by the reflections $s_\alpha:x\mapsto x-\frac{2\langle x,\alpha\rangle}{\langle \alpha,\alpha\rangle}\alpha$. The group $W$ also acts on $\mathfrak{a}$ by $wH_{\lambda}=H_{w\lambda}$ and is isomorphic to $N_K(\mathfrak{a})/Z_K(\mathfrak{a})$ (\cite[Thm. 6.57]{knapp2002lie}). By \cite[Thm. 4.3.24]{varadarajan2013lie}, this action can be extended to automorphisms of the Lie algebra $\mathfrak{g}$. In particular, we get that $\mathfrak{g}^{w\alpha}=w(\mathfrak{g}^\alpha)$ and so $m(w\alpha)=m(\alpha)$. The hyperplanes $\{\alpha(H)=0\}$ divide $\mathfrak{a}$ into $\vert W\vert$ connected components. We choose one, which we denote $\mathfrak{a}^+$ and call the positive Weyl chamber, and we define the positive roots $\Sigma^+=\{\alpha \in \Sigma \vert \forall H\in \mathfrak{a}^+,\alpha(H)>0\}$. Then $\Sigma=\Sigma^+ \cup (-\Sigma^+)$. We say that $\alpha\in \Sigma^+$ is simple if it cannot be decomposed as $\alpha=\beta+\gamma$ with $\beta,\gamma\in \Sigma^+$. Let $\Delta$ be the set of simple roots. Then $\Delta$ is a basis of $\mathfrak{a}^*$ and we can write $\Delta=\{\alpha_1,\cdots,\alpha_\ell\}$. Given $\alpha\in \Sigma^+$, $\alpha=\sum_{i=1}^\ell n_i(\alpha)\alpha$ with $n_i(\alpha)\in \N$. Furthermore, the group $W$ is generated by the reflections $\{s_\alpha\}_{\alpha\in \Delta}$ (\cite[Ch. VI, Thm. 2]{bourbaki2007groupes}). For any $\alpha\in \Delta$, the reflection $s_\alpha$ permutes the positive roots that are not proportional to $\alpha$ (\cite[Ch. VI, Prop. 17]{bourbaki2007groupes}).\smallskip

For $\lambda\in \mathfrak{a}^*$, define \begin{equation}\label{eq:nlambda}n(\lambda)=\sum_{\begin{subarray}{c}
    \alpha\in \Sigma^+\\
    \langle \alpha,\lambda\rangle\neq 0
\end{subarray}} m(\alpha)\end{equation}and set \begin{equation}
\label{eq:kappa}\kappa(G)=\underset{\lambda \in  \mathfrak{a}^*\setminus \{0\}}{\inf} \frac{n(\lambda)}{2}=\frac{1}{2}\underset{1\leq i \leq \ell}{\min}\underset{\begin{subarray}{c}
    \alpha\in \Sigma^+\\n_i(\alpha)\geq 1
\end{subarray}}{\sum} m(\alpha).\end{equation}

The values of $\kappa$ can be found in \cite[Appendix A]{dumas2024regularity}. It is also proven in \cite{dumas2024regularity} that $\kappa(G)$ is the optimal regularity of $K$-finite matrix coefficients of $G$.\smallskip

\begin{prop}[\cite{GindikinS.G.1967Urog}]\label{prop:spherfunctions}
    The positive-definite spherical functions of $(H,K)$ are $$\varphi_\lambda:X\mapsto \int_K e^{i\lambda(\Ad(k)X)}dk$$for $\lambda\in \mathfrak{a}^*$.
\end{prop}
The Killing form induces a scalar product $\langle,\rangle$ on $\mathfrak{p}$ and its subspace $\mathfrak{a}$. Then for $\lambda\in \mathfrak{a}^*$, there exists $H_\lambda\in \mathfrak{a}$ such that $\lambda(X)=\langle X,H_\lambda\rangle$. Furthermore, for any $k\in K$, we have $$\langle \Ad(k)X,Y\rangle=\langle X,\Ad(k^{-1})Y\rangle.$$Since $K$ is compact, we may use the change of variable $k\leftrightarrow k^{-1}$ to obtain the following expression for the spherical functions:

\begin{equation}
    \label{eq:sphversionDKV} \varphi_\lambda(X) = \int_K e^{i\langle X,\Ad(k)H_\lambda\rangle }dk
\end{equation}

\begin{remark}
    When $G=SO_0(n,1)$, $K=SO(n)$ acts on $\mathfrak{p}\simeq \R^n$ by its usual action. Thus, the Cartan motion group associated to $G$ is the Euclidean motion group. In this case, the spherical functions are Bessel functions (\cite[Ch. XI]{vilenkin1968special}).
    
    The integral formula in Proposition \ref{prop:spherfunctions} is a generalization of the integral formula for Bessel functions.
\end{remark}

\subsection{KAK decomposition}\label{sec:kak}
In this section, we prove the existence of a well-behaved $KAK$ decomposition for Cartan motion groups, in order to apply results from Section \ref{sec:kfinite}.

The following proposition is a consequence of well-known results in the theory of semisimple Lie groups, which are essentially the ingredients used for the $KAK$ decomposition at the level of the semisimple group $G$ itself.
\begin{prop}\label{prop:kakbase}
    Let $g\in H=\mathfrak{p}\rtimes K$. Then there exists a unique $a\in \overline{\mathfrak{a}^+}$ and such that $g\in K(a,\Id)K$. Furthermore, if $k_1,k_2$ are such that $$g=(0,k_1)(a,\Id)(0,k_2^{-1})$$and if $a\in \mathfrak{a}^+$ then $k_1$ is unique up to multiplication on the right by an element of $M=Z_K(\mathfrak{a})$.
\end{prop}

\begin{proof}
    Let $g=(x,k)$. Since $(0,k_1)(a,\Id)(0,k_2^{-1})=(\Ad(k_1)(a),k_1k_2^{-1})$, it suffices to show that for any $x\in \mathfrak{p}$, there exists a unique $a\in  \overline{\mathfrak{a}^+}$ such that $x\in \Ad(K)a$.

    \textit{Existence.} Since all maximal abelian subspaces of $\mathfrak{p}$ are conjugated under $K$, we have $\mathfrak{p}=\bigcup \Ad(k)\mathfrak{a}$ (\cite[Thm. 6.51]{knapp2002lie}). Thus, there exist $a_0\in \mathfrak{a}$ and $k_0\in K$ such that $x=\Ad(k_0)a_0$. Then by the theory of root systems, there exists $a\in \overline{\mathfrak{a}^+}$ the closure of the positive Weyl chamber and $w\in W$ such that $w.a=a_0$ (\cite[Prop. 8.29]{hall2003lie}). Since $W=N_K(\mathfrak{a})/Z_K(\mathfrak{a})$, $w$ is represented by $k_w\in K$ and $x=\Ad(k_0k_w)a$.

    \textit{Uniqueness.} Assume that $x=\Ad(k)a=\Ad(k')a'$ for $k,k'\in K$ and $a,a'\in \overline{\mathfrak{a}^+}$. Then $a,a'$ are conjugated under $K$ thus by \cite[Lemma 7.38]{knapp2002lie}, $a,a'$ are conjugated under $N_K(\mathfrak{a})$. So there is $w\in W$ such that $w.a=a'$ hence by \cite[Prop. 8.25]{hall2003lie}, $a=a'$.\smallskip

    Finally, the element $k$ such that $x=\Ad(k)a$ is unique up to multiplication on the right by an element of $Z_K(a)$. When $a\in \mathfrak{a}^+$, $Z_K(a)=M=Z_K(\mathfrak{a})$.
\end{proof}

By Proposition \ref{prop:kakbase}, we may define $P:H\to \overline{\mathfrak{a}^+}$ such that $P(g)$ is the only element of $\mathfrak{a}^+$ such that $g\in K(P(g),\Id)K$. Set $$H_r=\{g\in H \vert \ P(g)\in \mathfrak{a}^+\},$$this is a dense open subset of $H$ that we call the set of regular points of $H$. We now want to study the regularity of the decomposition.

\begin{lem}
    \label{lem:submersion}The map $$\fonction{q}{K\times K\times \mathfrak{a}^+}{H_r}{(k_1,k_2,a)}{(\Ad(k_1)a,k_1k_2^{-1})}$$is a submersion.
\end{lem}
\begin{proof}If $k\in K$, denote $L_k$ and $R_k$ the translations by $k$ on the left and right respectively. Let $m:G\times K\to K$ be the multiplication map, its differential at $(a,b)$ is $$\fonction{T_{(a,b)}m}{T_aK\times T_bK}{T_{ab}K}{(X_a,X_b)}{T_aR_b(X_a)+T_bL_a(X_b)}.$$We can identify $T_kK$ with $\mathfrak{k}$ by the isomorphism $T_eL_k$. Under this identification, we have $\forall g,h\in K$, $T_hL_g=\Id$ and $T_hR_g=\Ad(g^{-1})$, so that the tangent map becomes $T_{(a,b)}m(X_a,X_b)=\Ad(b^{-1})(X_a)+X_b$. Thus by the chain rule we have $$\fonction{T_{(k_1,k_2,a)}q}{\mathfrak{k}\times \mathfrak{k}\times \mathfrak{a}}{\mathfrak{p}\times \mathfrak{k}}{(X_1,X_2,Y)}{(\Ad(k_1)(Y+\ad(X_1)a),\Ad(k_2)(X_1)-X_2)}.$$We know that $\Ad(k)$ is an isomorphism of $\mathfrak{p}$. Thus, the map $T_{(k_1,k_2,H)}q$ is surjective if and only if $$\fonction{u}{\mathfrak{k}\times \mathfrak{a}}{\mathfrak{p}}{(X,Y)}{Y+[X,a]}$$ is surjective.
For $\alpha\in \Sigma^+$, let $\mathfrak{k}^\alpha=\mathfrak{k}\cap (\mathfrak{g}^\alpha \oplus \mathfrak{g}^{-\alpha})$ and $\mathfrak{p}^\alpha=\mathfrak{p}\cap (\mathfrak{g}^\alpha \oplus \mathfrak{g}^{-\alpha})$. From \cite[Ch. VI, Prop. 1.4]{loos1969symmetric2}, we get $$\mathfrak{k}=\mathfrak{m}\oplus \bigoplus_{\alpha\in \Sigma^+} \mathfrak{k}^\alpha=\mathfrak{m}\oplus \mathfrak{l},$$ $$\mathfrak{p}=\mathfrak{a}\oplus \bigoplus_{\alpha\in \Sigma^+} \mathfrak{p}^\alpha=\mathfrak{a}\oplus \mathfrak{b}.$$We also get that for $\alpha\in \Sigma^+$, there exists $Z_{\alpha,1},\cdots,Z_{\alpha,m(\alpha)}$ basis of $\mathfrak{g^\alpha}$, such that setting $Z_{\alpha,i}^{+}=Z_{\alpha,i}+\theta(Z_{\alpha,i})$ and $Z_{\alpha,i}^-=Z_{\alpha,i}-\theta(Z_{\alpha,i})$, $\{Z_{\alpha,i}^{+}\}$ is a basis of $\mathfrak{k}_\alpha$ and $\{Z_{\alpha,i}^{-}\}$ is a basis of $\mathfrak{p}_\alpha$.\\
Clearly, $\mathfrak{a}\subset \Ima u$. Furthermore, for $\alpha\in \Sigma^+,1\leq i\leq m(\alpha)$,  \begin{align*}
    u(Z_{\alpha,i}^+,0)&=[Z_{\alpha,i}^+,a]\\
    &=[Z_{\alpha,i},a]+[\theta Z_{\alpha,i},a]\\
    &=[Z_{\alpha,i},a]+\theta [Z_{\alpha,i},-a]\\
    &=-\alpha(a)Z_{\alpha,i}+\alpha(a)\theta Z_{\alpha,i}\\
    &=-\alpha(a)Z_{\alpha,i}^-.
\end{align*}
Since $a\in \mathfrak{a}^+$, $\alpha(a)\neq 0$ for all $\alpha\in \Sigma^+$ thus $Z_{\alpha,i}\in \Ima u$ and $u$ is surjective.
\end{proof}

\begin{prop}\label{prop:KAKversionlisse}The map $P:H_r\to \mathfrak{a}^+$ is smooth on $H_r$. Furthermore, for each $g\in H_r$, there exists a neighborhood $U_g$ of $g$ in $H_r$ and a choice of $g\mapsto k_i(g)$ such that $k_i$ is smooth on $U_g$, $i=1,2$ and for any $g\in U_g$, $$g=(0,k_1(g))(P(g),\Id)(0,k_2(g)^{-1}).$$
\end{prop}
In particular, Propositions \ref{prop:kakbase} and \ref{prop:KAKversionlisse} imply that $(H_r,K)$ has a well-behaved $KAK$ decomposition in the sense of Definition \ref{def:smoothkak}.

\begin{proof}Let $\Delta(M)=\{(m,m)\vert m\in M\}$ denote the diagonal subgroup of $K\times K$. By Lemma \ref{prop:kakbase}, the map $$\fonction{\Tilde{q}}{(K\times K)/\Delta(M)\times \mathfrak{a}^+}{H_r}{\left((k_1,k_2) \operatorname{ mod }M,a\right)}{(\Ad(k_1)a,k_1k_2^{-1})}$$is a well-defined smooth bijection between manifolds of the same dimension.\\
Let $p:K\times K\to (K\times K)/\Delta(M)$ be the projection. It is a surjective submersion. Let $q$ be the submersion defined in Lemma \ref{lem:submersion}, we have $q=\Tilde{q}\circ (p,\Id)$. Thus, for any $(x,a)\in (K\times K)/\Delta(M)\times \mathfrak{a}^+$, we have $T_{(x,a)}\Tilde{q}$ surjective. But it is a linear map between vector spaces of the same dimension, so it is invertible. Thus, by the local inversion theorem and since $\Tilde{q}$ is bijective, $\Tilde{q}$ is a smooth diffeomorphism.

Let $(x,P):G_r\to (K\times K)/\Delta(M)\times \mathfrak{a}^+$ be a smooth inverse. We get that $P$ is a smooth map. From \cite[Proposition 4.26]{lee2003introduction}, since $p$ is a submersion, any $(k_1,k_2)\in K\times K$ is in the image of a smooth local section of $p$. Let $g\in H_r$, since $p$ is surjective, $x(g)=p(k_1,k_2)$. There exists a neighborhood $V$ of $x(g)$ and a smooth section $s=(s_1,s_2):V\mapsto K\times K$ such that $s(x(g))=(k_1,k_2)$.\\
Let $U=x^{-1}(V)$ be a neighborhood of $g$, then $k_i=s_i\circ x$ is smooth on $U$ and $g=(0,k_1(g))(P(g),\Id)(0,k_2(g)^{-1})$.
\end{proof}

\section{Boundedness of spherical functions}\label{sec:mainresult}
\begin{theorem}\label{thm:cartanreg}
    Let $G$ be a connected semisimple Lie group with finite center and $H=\mathfrak{p}\rtimes K$ its Cartan motion group. Let $r=\lfloor \kappa(G)\rfloor$ and $\delta=\kappa(G)-r$. Then the family of positive-definite spherical functions of $(H,K)$ is bounded in $C^{(r,\delta)}(H_r)$.
\end{theorem}

\begin{proof}
    Up to composition with the smooth map $P$, it suffices to consider the group variable in $\mathfrak{a}^+$. Let $\lambda\in \mathfrak{a}^*$, $a\in \mathfrak{a}$ and $t\in \R$. Let $D$ be the differential operator with respect to the variable $a$. Denote $$f(\lambda,a,k,t)=e^{it\langle a,\Ad(k)H_\lambda\rangle}\in C^\infty(\mathfrak{a}^* \times \mathfrak{a}^+\times K\times \R).$$Then by induction on $s$, there is a polynomial $P\in  C^\infty(\mathfrak{a}^* \times \mathfrak{a}^+\times K\times \mathfrak{a}^s)[t]$ of degree $s$ such that for any $X\in \mathfrak{a}^s$, $$D^sf(\lambda,a,k,t)(X)=P(t)(\lambda,a,k,X)e^{it\langle a,\Ad(k)H_\lambda\rangle}.$$

    For $0\leq j \leq s$, let $g_j(\lambda,a,X)\in C^\infty(K)$ be defined by \begin{equation}\label{eq:defamplitude}g_j=\left.\frac{1}{j!} \frac{d^j}{dt^j}\left( (D^sf)e^{-it\langle a,\Ad(k)H_\lambda\rangle} \right)\right\vert_{t=0}.\end{equation}Then we have \begin{equation}\label{eq:diffspherique}D^s\varphi_{t\lambda}(a)(X)=\sum_{j=0}^s t^j \int_K e^{it\langle a,\Ad(k)H_\lambda\rangle}g_j(\lambda,a,X)(k)dk.\end{equation}

    We want to apply the method of stationary phase to understand the asymptotics of these integrals as $t\to \infty$. Let $f_{a,\lambda}:k\mapsto \langle a,\Ad(k)H_\lambda\rangle$ be the phase function. By \cite[Prop. 1.2]{DKV}, the critical set of $f_{a,\lambda}$ is $\mathcal{C}_\lambda=\bigcup_{w\in W} k_wK_\lambda$ where $K_\lambda=Z_K(\lambda)$ and $k_w$ is a representative of $w\in W$ (since $M=Z_K(\mathfrak{a})\subset K_\lambda$ this does not depend on the choice of $k_w$). By \cite[Prop. 1.4]{DKV}, the Hessian of $f_{a,\lambda}$ at $k_wm$, $m\in K_\lambda$, is diagonal in the basis adapted to the orthogonal decomposition of $\mathfrak{p}=\mathfrak{m}\oplus \bigoplus \mathfrak{p}_\alpha$, where $\mathfrak{p}_\alpha=\mathfrak{p}\cap \left(\mathfrak{g}_\alpha\oplus \mathfrak{g}_{-\alpha}\right)$, with eigenvalues $-\alpha(H_\lambda)(w\alpha)(a)$ in each $\mathfrak{p}_\alpha$.

    We now proceed as in \cite[Prop. 9.2]{DKV}. Let $S$ be the unit sphere in $\mathfrak{a}^*$ and $L$ a compact subset of $\mathfrak{a}^+$. Fix $\lambda_0\in S,a_0\in L$. Let $k_0\in \mathcal{C}_{\lambda_0}$, we can find a chart $$\fonction{\psi}{U_{k_0}}{U\times V\subset\R^{\dim \mathcal{C}_{\lambda_0}}\times \R^{n(\lambda_0)}}{k}{(x,y)}$$ around $k_0$ such that $\psi(k_0)=(0,0)$ and $\mathcal{C}_\lambda$ is given in these local coordinate by $\{y=0\}$. Treat $f_{a,\lambda}(k)$ as a function of $y$ with parameters $a,\lambda,x$. Then $f$ has a unique nondegenarate critical point $y=0$ at $(a,\lambda,x)=(a_0,\lambda_0,0)$. Thus, by \cite[Thm. 7.7.6]{hormander1983analysis}, there exists a neighborhood $U_{a_0}\times U_{\lambda_0}\times U_0$ of $(a_0,\lambda_0,0)$ and a continuous seminorm $\nu$ on $C^\infty(K)$, such that for any $a\in U_{a_0}$, $\lambda \in U_{\lambda_0}$, $x\in U_0$, $g\in C^\infty(K)$ and $t\geq 1$, \begin{equation}
        \label{eq:phasestat1} \left\vert \int_V e^{itf_{a,\lambda}(\psi^{-1}(x,y))}g(\psi^{-1}(x,y))dy \right\vert \leq \nu(g)t^{-n(\lambda_0)/2}\leq \nu(g)t^{-\kappa(G)}.
    \end{equation}By compactness we may cover $L\times S$ by a finite number of open subset $U_{a_0}\times U_{L_0}$ so up to changing the seminorm and $U_0$, we may assume \eqref{eq:phasestat1} holds for any $a\in L,\lambda\in S$. Similarly by compactness, we can cover $K$ by a finite number of open subset of the form $U_0\times V$. Using a smooth partition of unity and summing estimates in \eqref{eq:phasestat1}, we obtain a seminorm $\nu$ such that for any $a\in L$, $\lambda \in S$, $g\in C^\infty(K)$ and $t\geq 1$, \begin{equation}
        \label{eq:phasestat2} \left\vert \int_K e^{itf_{a,\lambda}(k)}g(k)dk \right\vert \leq \nu(g)t^{-\kappa(G)}.
    \end{equation}
    Now if $a\in L,\lambda\in S$ and $t\leq 1$, then by \eqref{eq:diffspherique} we get that \begin{equation}
        \label{eq:smallt} \Vert D^s\varphi_{t\lambda}(a)\Vert  \leq \underset{a\in L,\lambda\in S,k\in K\Vert X_i\Vert=1}{\sup} \sum \vert g_j(\lambda,a,X)(k)\vert \leq C_{L,s}
    \end{equation}where $C_{L,s}>0$ depends only on the compact $L$ and the order of differentiation.

    On the other hand for $t\geq 1$, combining \eqref{eq:diffspherique} and \eqref{eq:phasestat2}, we obtain \begin{equation}
        \label{eq:larget} \begin{aligned}
        \Vert D^s\varphi_{t\lambda}(a)\Vert  & = \underset{\Vert X_i\Vert=1}{\sup} \vert D^s\varphi_{t\lambda}(a)(X)\vert  \\
         & \leq  \underset{\Vert X_i\Vert=1}{\sup}\sum_{j=0}^s t^j \left\vert\int_K e^{itf_{a,\lambda}(k)} g_j(\lambda,a,X)(k)dk\right\vert \\
         & \leq \sum_{j=0}^s t^j\underset{\begin{subarray}{c}
  \Vert X_i\Vert=1\\
  a\in L,\lambda\in S
  \end{subarray}}{\sup}\nu(g_j(\xi,\eta,Y,X))t^{-\kappa(G)}\\
  & \leq  D_{L,s}t^{s-\kappa(G)} 
    \end{aligned}
    \end{equation}
where $D_{L,s}>0$ depends again only on the compact $L$ and the order of differentiation.

Thus combining \eqref{eq:smallt} and \eqref{eq:larget}, for any $\lambda\in \mathfrak{a}^*$, $a\in L$, $s\leq r$, \begin{equation}\label{eq:finalineq_integer}\Vert D^s\varphi_{\lambda}(a)\Vert \leq \max(C_{L,s},D_{L,s})=M_{L,s}.\end{equation}

Thus the differentials of the family of spherical functions are bounded independently on $\lambda$ up to order $r$. If $r=\kappa(G)$, the proof is complete.\smallskip

Otherwise, $\kappa(G)-r=\frac{1}{2}$. Then using \eqref{eq:larget} for $s=r$ and $s=r+1$, we show that for any $x,y\in L$, $\xi\in S$,$\eta\in C$, $t\geq 1$, we have on the one hand \begin{equation}\label{eq:cartanholderstep1}\Vert D^r \varphi_{t\lambda}(x)-D^r\varphi_{t\lambda}(y)\Vert \leq \Vert D^r \varphi_{t\lambda}(x)\Vert+\Vert D^r\varphi_{t\lambda}(y)\Vert\leq 2D_{L,r}t^{-1/2}\end{equation}and on the other hand, we get by the mean value theorem that \begin{equation}\label{eq:cartanholderstep2}
    \Vert D^r \varphi_{t\lambda}(x)-D^r\varphi_{t\lambda}(y)\Vert \leq \left(\underset{a\in L}{\sup}\Vert {D^{r+1}}\varphi_{t\lambda}(a)\Vert\right)\Vert x-y\Vert \leq D_{L,r+1}t^{1/2}\Vert x-y\Vert.\end{equation}
Thus, combining \eqref{eq:cartanholderstep1} and \eqref{eq:cartanholderstep2} yields \begin{equation}\label{eq:cartanholderstep3}\Vert D^r \varphi_{t\lambda}(x)-D^r\varphi_{t\lambda}(y)\Vert \leq \left(2D_{L,r}D_{L,r+1}\right)^{1/2}\Vert x-y\Vert^{1/2}.\end{equation}
Hence, setting $M_L=\max\left(\left(2D_{L,r}D_{L,r+1}\right)^{1/2},C_{L,r+1}(\diam L)^{1/2}\right)$, we have that for any $\lambda\in \mathfrak{a}^*$, $x,y\in L$, \begin{equation}\label{eq:holderstepfin}\Vert D^r \varphi_{\lambda}(x)-D^r\varphi_\lambda(y)\Vert \leq M_L \Vert x-y\Vert^{\kappa(G)-r}.\qedhere\end{equation}
\end{proof}

We will now show that this result is optimal. We will use the following lemma (\cite[Lemma 4.5]{dumas2024regularity}) which essentially states that families of exponential maps are not bounded in Hölder spaces.
\begin{lem}\label{lem:expo}Let $E$ be a finite dimensional real vector space, $U$ an open subset of $E$ such that $0\in \overline{U}$. Let $u_1,\cdots,u_n\in E^*$ distinct and non-zero, and $f_1,\cdots,f_n:E\to \C$ continuous functions such that for any $U'$ open subset of $U$, there is $x\in U'$ such that $\sum \vert f_j(x)\vert\neq 0$. Then there exists $C>0$, $d>0$, $x\in U$, and an open set $V$ with $0\in \overline{V}$ such that for all $y=x+h,h\in V$, $m\in \N$ and $N\geq \frac{d}{\Vert h\Vert}$, $$\frac{1}{N}\sum_{t=m}^{m+N-1} \left\vert \sum_{j=1}^n  f_j(x)e^{itu_j(x)}-f_j(y)e^{itu_{j}(y)} \right\vert^2 \geq C.$$
\end{lem}

In order to prove optimality, we will show that a particular subfamily of the whole family of positive-definite spherical functions is already unbounded in Hölder spaces of higher order.
\begin{theorem}\label{thm:cartanopti}Let $G$ be a connected semisimple Lie group with finite center and $H=\mathfrak{p}\rtimes K$ its Cartan motion group. Let $r=\lfloor \kappa(G)\rfloor$ and $\delta=\kappa(G)-r$. For any $\delta'>\delta$, the family of positive definite spherical functions of $(H,K)$ is not bounded in $C^{(r,\delta')}(H_r)$.
\end{theorem}

\begin{proof}
Fix $\lambda\in \mathfrak{a}^*$ such that $n(\lambda)=2\kappa(G)$, and such that $\langle\alpha,\lambda\rangle \geq 0$ for any $\alpha\in \Sigma^+$. It suffices to show that the family $\left(\varphi_{t\lambda}\right)_{t\in \R}$ is not bounded in $C^{(r,\delta'}(\mathfrak{a}^+)$. As in the proof of Theorem \ref{thm:cartanreg}, let $f_{a,\lambda}:k\mapsto \langle a,\Ad(k)H_\lambda\rangle$ be the phase function. Since $f_{a,\lambda}$ is right $K$-invariant, we may see the integral on the quotient manifold $K/K_\lambda$. More precisely, let $d(k_\lambda)$ be the image of the Haar measure on $K$ under the quotient map $K\to K/K_\lambda$, then for all $t\geq 1$, $$\varphi_{t\lambda}(a)=\int_{K/K_\lambda} e^{itf_{a,\lambda}(k)}d(kK_\lambda).$$This will simplify the proof, as the critical set of $f_{a,\lambda}$ is thus $\mathcal{C}=\bigcup_{w\in W}k_wK_\lambda$ a finite set in $K/K_\lambda$. As in \eqref{eq:diffspherique}, for any $a\in \mathfrak{a}$, $t\geq 1$, $X\in \mathfrak{a}^r$, \begin{equation}
    \label{eq:cartanderivphicasopti} D^r\varphi_{t\lambda}(a)(X)=\sum_{j=0}^r t^j \int_{K/K_\lambda} e^{itf_{a,\lambda(k)}} g_j(a,X)(k)d(kK_\lambda).
\end{equation}
Let $I_j(a,X,t)=\int_K e^{itf_{a,\lambda}(k)} g_j(a,X)(k)d(kK_\lambda)$. Let $W_\lambda$ denote the stabilizer of $\lambda$ under the action of the Weyl group $W$. Let also $$\Sigma^+(\lambda)=\{\alpha \in \Sigma^+ \vert \langle \alpha,\lambda\rangle\neq 0\}$$ and $$\sigma_w=-\underset{\alpha\in \Sigma^+(\lambda)\neq 0}{\sum} m(\alpha)\operatorname{sgn}(\langle\alpha,\lambda\rangle)(w\alpha)(Y)).$$Let $d_0k$ denote the Riemannian measure on $K$ induced by the bi-invariant metric defined by the Killing form on $\mathfrak{k}$.  Let $d_0(kK_\mu)$ be the volume measure on $K/K_\mu$ associated to the (invariant) Riemannian metric induced by the restriction of the inner product $-\langle\cdot,\cdot\rangle$ on $\mathfrak{k}$. Denote $\operatorname{Vol}(K/K_\mu)=\int_{K/K_\mu} d_0(K/K_\mu)$. By uniqueness of the invariant measure on $K/K_\mu$, we have $d(kK_\mu)=\frac{1}{\operatorname{Vol}(K/K_\mu)}d_0(kK_\mu)$. For $w\in W$, $g\in C^\infty(K/K_\lambda)$, $a\in \mathfrak{a}^+$, set \begin{equation}
    \label{eq:cartandistrib} c_{w,a}(g)=e^{i\frac{\pi}{4}\sigma_w} \prod_{\alpha\in \Sigma^{+}(\lambda)}\left\vert\frac{\langle \alpha,\lambda \rangle}{4\pi}\left(w\alpha)(a)\right)\right\vert^{-\frac{m(\alpha)}{2}} \frac{1}{\operatorname{Vol}(K/K_\lambda)}g(k_wK_\lambda).
\end{equation}
Then, by the stationary phase approximation (\cite[Thm. 7.7.6]{hormander1983analysis},\cite[Thm. 5.5]{dumas2024regularity}), for any $a\in \mathfrak{a}^+$, there is a neighbourhood $U_a$ of $a$ in $\mathfrak{a}^+$ and $D(a)>0$ such that for any $0\leq j\leq r$, $t\geq 1$, $a'\in U_a$ and $X$ with $\Vert X_i\Vert=1$ for all $i$, \begin{equation}
    \label{eq:cartanphasestatopti}\left\vert I_j(a',X,t) - \sum_{W/W_\lambda} e^{it(w\lambda)(a')}t^{-\kappa(G)}c_{w,a'}(g_j(a',X))\right\vert \leq D(a)t^{-{\kappa(G)}-1}.
\end{equation}We use that $g_j$ is smooth in all variables hence bounded on compact sets and that the bound is uniform in the parameter $a'$ of the phase function.\smallskip

The end of the proof relies on Lemma \ref{lem:expo} and is similar to \cite[Thm. 4.4]{dumas2024regularity}.\smallskip

Combining \eqref{eq:cartanderivphicasopti} with \eqref{eq:cartanphasestatopti} for $0\leq j< r$, for any $a$ there is a neighbourhood $V_a$ of $a$ and a constant $C(a)$ such that for any $t\geq 1$, $a'\in V_a$ and $X$ with $\Vert X_i\Vert=1$, \begin{equation}\label{eq:cartanstep1}
    \left\vert D^r\varphi_{t\lambda}(a')(X)-t^r I_r(a',X,t) \right\vert \leq C(a)t^{-1}.
\end{equation}

For $X$ fixed with $\Vert X_i\Vert=1$, let $S_t(x)=\sum_{W/W_\lambda} e^{it(w\lambda)(x)}c_{w,x}(g_r(x,X))$. Combining \eqref{eq:cartanstep1} and \eqref{eq:cartanphasestatopti}, if $t\geq 1$, and $x,y\in U_a\cap V_a$, \begin{equation}\label{eq:cartanstep2}\begin{aligned}
    t^{-\delta}\vert S_t(x)-S_t(y)\vert & \leq  t^r\vert t^{-\kappa(G)}S_t(x)-I_r(x,X,t)\vert + t^r\vert I_r(x,X,t)-I_r(y,X,t)\vert \\ 
    &\phantom{\leq}+ t^r\vert I_r(y,X,t)-t^{-\kappa(G)}S_t(y)\vert  \\
     & \leq  2D(a)t^{-\delta-1}+\vert t^rI_r(x,H,t)-D^r\varphi_{t\lambda}(x)(X) \vert \\
     &\phantom{\leq} + \vert D^r\varphi_{t\lambda}(x)(X)-D^r\varphi_{t\lambda}(y)(X) \vert\\
     &\phantom{\leq}+\vert D^r\varphi_{t\lambda}(y)(X)-t^rI_r(y,X,t) \vert\\
     &\leq  2D(a)t^{-\delta-1}+2C(a)t^{-1} + \vert D^r\varphi_{t\lambda}(x)(X)-D^r\varphi_{t\lambda}(y)(X) \vert\\
     & \leq  \Vert D^r\varphi_{t\lambda}(x)-D^r\varphi_{t\lambda}(y) \Vert + 2(C(a)+D(a))t^{-1}.
\end{aligned}\end{equation}Now the functions $c_{w,x}(g_r(x,X))$ are all zero at $x$ if and only if $g_r(x,X)(k_wK_\lambda)=0$ for all $w\in W$. But $$g_r(x,X)(kK_\lambda)=\prod_{i=1}^r \langle X_i,\Ad(k)H_\lambda\rangle.$$Thus given any open subset of $\mathfrak{a}^+$, we can choose $x,X$ such that $g_r(x,X)(e)\neq 0$. Thus the hypotheses of Lemma \ref{lem:expo} hold for the family of functions $f_w:x\mapsto c_{w,x}(g_r(x,X))$, for $U=\mathfrak{a}^+$. Let $C,d,x,V$ be given by Lemma \ref{lem:expo}, $W_x=x+V$ such that for any $y\in W_x$, $m\in \N$, $N\geq \frac{d}{\Vert x-y\Vert}$, \begin{equation}
    \label{eq:cartanstep3} \sum_{t=m}^{m+N-1} \vert S_t(x)-S_t(y)\vert^2 \geq CN.
\end{equation} From now on, we choose $a=x$ given above. Let $M=4(C(x)+D(x))^2$, we get from \eqref{eq:cartanstep2} that for any $t\geq 1$, $y\in U_x\cap V_x\cap W_x$, \begin{equation}
    \label{eq:cartanstep4}\frac{t^{-2\delta}}{2}\vert S_t(x)-S_t(y)\vert^2 \leq \Vert D^r\varphi_{t\lambda}(x)-D^r\varphi_{t\lambda}(y) \Vert^2+ Mt^{-2}
\end{equation}

Assume now that the family of positive definite spherical functions of $(H,K)$ is bounded in $C^{(r,\delta')}(\mathfrak{a}^+)$ for $\delta'> \delta$. In particular, up to reducing $U_x\cap V_x\cap W_x$ to a bounded subset of diameter $L$ if necessary, there is $D>0$ such that for any $y\in U_x\cap V_x\cap W_x$ and $t\geq 1$, \begin{equation}
    \label{eq:cartanhypo} \Vert D^r\varphi_{t\lambda}(x)-D^r\varphi_{t\lambda}(y) \Vert \leq D\Vert x-y\Vert^{\delta'}
\end{equation}

For $y$ fixed, set $m,N$ such that \begin{equation}\label{eq:choicem}\frac{1}{\Vert x-y\Vert^{\delta'}}\leq m\leq \frac{1}{\Vert x-y\Vert^{\delta'}}+1\end{equation}and \begin{equation}\label{eq:choiceN}\frac{d}{\Vert x-y\Vert}\leq N\leq \frac{d}{\Vert x-y\Vert}+1.\end{equation}Combining \eqref{eq:cartanstep3}, \eqref{eq:cartanstep4} and \eqref{eq:cartanhypo} gives \begin{equation}
    \label{eq:cartanstep5} \begin{aligned}\frac{CN}{2(m+N)^{2\delta}}&\leq\sum_{t=m}^{m+N-1} \frac{t^{-2\delta}}{2} \vert S_t(x)-S_t(y)\vert^2 \\&\leq \sum_{t=m}^{m+N-1} \left(\Vert D^r\varphi_{t\lambda}(x)-D^r\varphi_{t\lambda}(y) \Vert^2+ Mt^{-2}\right)\\ &\leq ND^2\Vert x-y\Vert ^{2\delta'}+\frac{MN}{m^2}\end{aligned}
\end{equation}
thus \begin{equation}
    \label{eq:cartanstep6} \frac{C}{2(m+N)^{2\delta}} \leq D^2\Vert x-y\Vert ^{2\delta'}+\frac{M}{m^2} \leq (D^2+M)\Vert x-y\Vert^{2\delta'}
\end{equation}by \eqref{eq:choicem}. But by \eqref{eq:choicem} and \eqref{eq:choiceN}, we have \begin{multline}
    m+N\leq \frac{d}{\Vert x-y\Vert}+1+\frac{1}{\Vert x-y\Vert^{\delta'}}+1 \leq \frac{1}{\Vert x-y\Vert}\left(d+2\Vert x-y\Vert +\Vert x-y\Vert^{1-\delta'}\right)\\ \leq \frac{1}{\Vert x-y\Vert}\left(d+2L+L^{1-\delta'}\right)
\end{multline}hence \eqref{eq:cartanstep6} becomes \begin{equation}
    \label{eq:step7} \frac{C}{2(d+2L+L^{1-\delta'})^{2\delta}}\Vert x-y\Vert^{2\delta}\leq (D^2+M)\Vert x-y\Vert^{2\delta'}.
\end{equation}
Since \eqref{eq:step7} holds for any $y\in U_x\cap V_x\cap W_x$ with the constant involved independent from $y$ and $\delta'>\delta$, we get a contradiction as $y$ goes to $x$ (which is possible because $0\in \overline{V}$ hence $x\in \overline{U_x\cap V_x\cap W_x}$).    
\end{proof}

Finally, we can prove the main theorem as a consequence of these results.
\begin{theorem}
    Let $G$ be a connected semisimple Lie group with finite center and $H=\mathfrak{p}\rtimes K$ its Cartan motion group. Let $r=\lfloor \kappa(G)\rfloor$ and $\delta=\kappa(G)-r$. Then any $K$-finite matrix coefficient of a unitary representation of $H$ is in $C^{(r,\delta)}(H_r)$. Furthermore, this result is optimal in the sense that  for any $\delta'>\delta$, there exists a $K$-finite (even $K$-bi-invariant) coefficient which is not in $C^{(r,\delta')}(H_r)$.
\end{theorem}
\begin{proof}
    By Propositions \ref{prop:kakbase} and \ref{prop:KAKversionlisse}, $(H_r,K)$ has a well-behaved $KAK$ decomposition. Thus by Theorem \ref{thm:kfinite}, it suffices to consider $K$-bi-invariant matrix coefficients. The result is then an immediate consequence of Lemma \ref{lem:lienspheriquekbiinv} and Theorems \ref{thm:cartanreg} and \ref{thm:cartanopti}.
\end{proof}

\begin{remark}
    Given any compact Lie group and any orthogonal representation $\pi:K\to O(V)$ on a finite dimensional real vector space $V$, we can form the semi-direct product $H=V\rtimes K$. In general, there is no reason for $H$ to be the Cartan motion group of a semisimple Lie group. However, $(H,K)$ remains a Gelfand pair, so we may ask if we can find the optimal regularity of $K$-finite matrix coefficients of unitary representations of $H$. By Mackey's theory (\cite{mackeysemidirect}), we can still find the irreducible representations of $H$ with a nonzero $K$-invariant vector and thus compute the positive definite spherical functions. Indeed, $K$ acts on $\hat{V}$ the space of characters of $V$. Given some $\chi_\lambda:x\mapsto e^{i\langle \lambda,x\rangle}\in \hat{V}$, let $K_\lambda=\Stab(\chi_\lambda)=\Stab (\lambda)$ and $H_\lambda=V\rtimes K_\lambda$. Then for any $(\rho,V_\rho)\in \hat{K}_\lambda$ irreducible representation of $K_\lambda$, we may consider the induced representation $\mathrm{Ind}_{H_\lambda}^H (\chi_\lambda \otimes \rho)$ acting on $L^2(H/H_\lambda;V_\rho)$. Mackey's result is that those representations are irreducible, and that they exhaust the irreducible representations of $H$. It is not difficult to see that the induced representation has a $K$-invariant vector if and only if $\chi_\lambda\otimes \rho$ has a $K_\lambda$-invariant vector, thus by irreducibility of $\rho$, if and only if $\rho$ is the trivial representation. In that case, the one-dimensional space of invariant vectors is the space of constant functions. Thus, the spherical function associated to $\mathrm{Ind}_{H_\lambda}^H (\chi_\lambda \otimes 1)$ can be expressed by $$\varphi_\lambda(v)=\int_K e^{i\langle \lambda,\pi(k)v\rangle}dk$$with $\lambda\in V$, and $\varphi_\lambda=\varphi_{\lambda'}$ if $\lambda,\lambda'$ are in the same $K$-orbit. So the spherical functions are still oscillatory integrals. However, the lack of semisimple structure makes it difficult to understand the orbits and the critical points of the phase functions.

    The simplest case to consider is $K=SO(3)$ and $\pi$ irreducible. In the case of Cartan motion groups, the representation on $\mathfrak{p}$ is irreducible if and only if the ambient semisimple Lie group $G$ is simple. By the classification of simple Lie groups, we see that the only irreducible representations of $SO(3)$ which appear in Cartan motion groups of semisimple Lie groups are the $3$-dimensional and the $5$-dimensional representations. Thus, the first case outside the world of Cartan motion groups to consider seems to be the $7$-dimensional representation of $SO(3)$ on the space $V_3$ of harmonic homogeneous polynomials of degree $3$ in $3$ variables with real coefficients. In this specific case, we may be able to make the necessary computations without the additional semisimple structure.
\end{remark}

\nocite{*}

\bibliographystyle{alpha}
\bibliography{Ref}
\end{document}